\documentclass[11pt,oneside]{article}

\usepackage{amsthm, amsmath, amssymb, amsfonts,epsfig}
\usepackage{mathtools}
\usepackage{esint}
\usepackage{epstopdf}

\usepackage{graphicx}
\usepackage[font=sl,labelfont=bf]{caption}
\usepackage{subcaption}

\usepackage{enumerate}

\usepackage[colorlinks=true, pdfstartview=FitV, linkcolor=blue,
            citecolor=blue, urlcolor=blue]{hyperref}
\usepackage[usenames]{color}
\definecolor{Red}{rgb}{0.7,0,0.1}
\definecolor{Green}{rgb}{0,0.7,0}

\usepackage[numbers]{natbib}

\usepackage{url}
\makeatletter\def\url@leostyle{%
 \@ifundefined{selectfont}{\def\UrlFont{\sf}}{\def\UrlFont{\scriptsize\ttfamily}}} \makeatother

\usepackage{accents, comment}

\usepackage{bbm}

\usepackage{mathrsfs}
\usepackage{leftidx}

\newtheorem{theorem}{Theorem}

\newtheorem{lemma}[theorem]{Lemma}

\theoremstyle{definition}
\newtheorem{definition}[theorem]{Definition}

\theoremstyle{remark}
\newtheorem{remark}[theorem]{Remark}

\numberwithin{equation}{section}
\numberwithin{theorem}{section}

\def\cC{\mathcal{C}}
\def\cD{\mathcal{D}}

\def\cM{\mathcal{M}}

\def\bC{\mathbb{C}}

\def\bR{\mathbb{R}}

\newcommand{\1}{\mathbbm{1}}                     

\title{A Regularity Criterion for Solutions to the 3D NSE\\ in `Dynamically Restricted' Local Morrey Spaces}

\author{
Zoran Gruji\'c \\
\and 
 Liaosha Xu\\
}

\begin{document}

\maketitle

\begin{abstract}

It is shown that a local-in-time strong solution $u$ to the 3D Navier-Stokes equations remains regular on an interval $(0,T)$ provided a smallness $\epsilon_0$-condition on $u$ in a lower time-restricted local Morrey space is stipulated; more precisely,
\[
\sup_{t\in(0,T)}\ \sup_{x\in\bR^3,\ \eta(t)\le r\le1}\ \frac{1}{r^\alpha}\int_{B_r(x)}|u(y,t)|^pdy \le \epsilon_0
\]
where $\eta$ is a dynamic dissipation scale consistent with the turbulence phenomenology and $\alpha$ and $p$ are suitable parameters. Such regularity criterion guarantees the volumetric sparseness of local spatial structure of intense vorticity components, preventing the formation of the finite-time blow up at $T$ under the framework of $Z_\alpha$-sparseness classes introduced in \citet{Bradshaw2019}.

\end{abstract}

\section{Introduction}

Consider the 3D incompressible, visous fluid modeled by the 3D Navier-Stokes equations (NSE)
\begin{align*}
u_t-\Delta u+ u\cdot\nabla u+\nabla p=0\ ,\qquad \textrm{div}\ u=0\ ,\qquad u(x,0)=u_0\ .
\end{align*}
For simplicity, set the viscosity $\nu$ to be 1 and the external force $f$ to be 0. Henceforth, the spatial domain will be the whole space $\bR^3$.

Morrey spaces, both global and local variants (or--in other words--nonrestricted and restricted versions), 
have been utilized in the mathematical study of the 3D NSE as a source of the initial configurations yielding local-in-time well-posedness for large initial data and global-in-time well-posedness for small initial data (e.g., \citet{Taylor1992} and \citet{Giga1989}), as well as the mathematical framework suitable for formulating interior regularity conditions (e.g., \citet{Chen2001}). The local Morrey spaces are of particular interest as they describe uniformly-local behavior of local $L^p$-quantities with respect to the (physical) scale. Local conditions of the Caffarelli-Kohn-Nirenberg (CKN)-type ($\limsup_{r\to0}$ on the parabolic cylinder-type) are closely related (\citet{Caffarelli1982}).

In 1938, \citet{Morrey1938} introduced the following norms,
\begin{align*}
\|f\|_{M_{p,\lambda}}:=\sup_{B_r(x)} r^{-\lambda/p} \|f\|_{L^p(B_r(x))}.
\end{align*}
In particular, denote by $\cM_q$ the class of locally integrable functions on $\bR^3$ satisfying
\begin{align*}
\underset{x\in\bR^3,\ 0<r<\infty}{\sup}\ r^{3(1/q-1)}\int_{B_r(x)} |f(y)|dy<\infty\ ;
\end{align*}
one can also formulate an analogous condition in terms of the locally finite measures (e.g., \citet{Giga1989} obtained well-posedness results for the vorticity formulation of the 3D NSE in the class of locally finite measures with the exponent $q=\frac{3}{2}$).

The local Morrey-type spaces $LM_{p,\theta,w}$ were introduced by \citet{Guliev1998} and the norm is given by
\begin{align*}
\|f\|_{LM_{p,\theta,w}}:=\|w(r)\|f\|_{L^p(B_r(0))}\|_{L^\theta(0,\infty)}\ 
\end{align*}
where $w$ is a positive measurable function defined on $(0,\infty)$. Later, Guliyev considered the complementary local Morrey-type space and the norm is given by
\begin{align*}
\|f\|_{{}^{\cC}\!LM_{p,\theta,w}}:=\|w(r)\|f\|_{L^p(\bR^d\setminus B_r(0))}\|_{L^\theta(0,\infty)}\ .
\end{align*}
The global Morrey-type space $GM_{p,\theta,w}$ was introduced in \citet{Burenkov2004} with
\begin{align*}
\|f\|_{GM_{p,\theta,w}}:=\sup_{x\in\bR^d}\|w(r)\|f\|_{L^p(B_r(x))}\|_{L^\theta(0,\infty)}\ .
\end{align*}
Note that $GM_{p,\infty,r^{-\lambda}}=M_{p,\lambda}$. The precise definitions of these generalized Morrey-type spaces will be given in Section~\ref{sec:MorreyPreDual}.

Classical local Morrey spaces correspond to the weight function $w(r)=r^{-\alpha}$ and the range of scales $r$ in $(0,1]$,
\begin{align*}
\underset{x\in\bR^3,\ 0<r\le1}{\sup}\ r^{-\alpha}\int_{B_r(x)} |f(y)|^pdy<\infty;
\end{align*}
in particular, this leads to a local, scaling-invariant 3D NSE regularity class 
$$L^\infty\left((0,T); GM_{2,\infty,r^{-\frac{1}{2}}}\right),$$
i.e., a requirement that
\begin{align}\label{eq:MorrRegCond}
\sup_{t\in(0,T)}\ \sup_{x\in\bR^3,\ 0<r\le1}\ \frac{1}{r}\int_{B_r(x)}|u(y,t)|^2dy \le \epsilon_0
\end{align}
for a suitably small positive constant $\epsilon_0$. This condition is also of interest from the physical point of view as it rigorously quantifies the uniformly-local $r$-behavior of the local kinetic energy of the fluid consistent with the scaling.

A natural question motivated by the physics of turbulence is whether it is possible to restrict the range of scales in the local Morrey-type conditions to an interval of the form $[\eta(t),1]$ for some explicit, positive function $\eta$ representing a manifestation of the dissipation scale. A recent rigorous work by \citet{Bradshaw2017} presented a conceptually analogous result in the setting of the Besov spaces in which the physical scales are replaced by the Littlewood-Paley frequencies, supplementing the physical motivation. The goal of this project is to show that it is indeed possible to replace/weaken the classical local Morrey conditions with the suitably defined lower time-restricted versions. In particular, the claim is that the Morrey regularity condition \eqref{eq:MorrRegCond} can be replaced with the following dynamically restricted Morrey condition,
\begin{align}\label{eq:RestrMorrCond}
\sup_{t\in(0,T)}\ \sup_{x\in\bR^3,\ \|\nabla u(t)\|_\infty^{-\frac{1}{2}}\lesssim r\le1}\ \frac{1}{r}\int_{B_r(x)}|u(y,t)|^2dy \le \epsilon_0\ .
\end{align}
The idea is to show that the condition \eqref{eq:RestrMorrCond} implies the vorticity $\omega$ eventually (sufficiently close to the blow-up time) gets trapped in the critical sparseness class $Z_{\frac{1}{2}}$ introduced in \citet{Bradshaw2019}, yielding a contradiction. Membership in $Z_{\frac{1}{2}}$ will follow from modifying a criterion for 3D sparseness at scale $r$ stating that if
\begin{align}
\|f\|_{H^{-1}}\le c^*(\lambda,\delta)r^{\frac{5}{2}}\|f\|_\infty
\end{align}
with $\lambda,\delta,c^*$ properly chosen, then each of the six super-level sets
\begin{align}\label{eq:Zsparse}
S_\lambda^{i,\pm}=\left\{x\in\bR^3 :\ f_i^\pm(x)>\lambda\|f\|_\infty\right\},\qquad 1\le i\le3
\end{align}
is 3D $\delta$-sparse at scale $r$ (see \citet{Bradshaw2019}); here, $H^{-1}$ denotes the dual of the Sobolev space $H^1$.
In order to utilize the uniformly-local condition \eqref{eq:RestrMorrCond} within the sparseness realm, a technical adjustment of the global $H^{-1}$-norm will be required. 

Furthermore, following the general idea for establishing the pre-duality results exposed in \citet{Gogatishvili2011}, we obtain the regularity of solutions under more general assumption than \eqref{eq:RestrMorrCond}, namely,
\begin{align*}
\sup_{t\in(0,T)}\ \|u(t)\|_{GM_{p,\theta,w}} \le \epsilon_0
\end{align*}
with $w(r)=r^{-\beta}\1_{\left[\|\omega(t)\|_\infty^{-\delta} ,\ 1\right]}$, discovering that $\delta$, the exponent on the lower restriction, is intrinsically related to the $\alpha$-parameter of the regularity class $Z_\alpha$. 
This paper is also a demonstration of how the $Z_\alpha$ formalism in \citet{Bradshaw2019} can be utilized to obtain novel results in the setting of the more traditional functional frameworks.

The paper is organized as follows. In section~\ref{sec:SparseRegAssm}, we recall the notion of sparseness and the associated geometric-type criteria for smoothness of solutions to the 3D NSE. In section~\ref{sec:MorreyPreDual}, we give the precise definition of the generalized Morrey-type spaces, $GM_{p,\theta,w}$ and exhibit some results about the pre-duality in this type of spaces as well as prove a key technical lemma needed for our main results. Section~\ref{sec:MainThm} is devoted to formulating the main theorems and demonstrating the proofs relying on the results gathered in the previous sections.

\section{Sparseness and Geometric Regularity Criteria}\label{sec:SparseRegAssm}

To be able to give a precise statement of the main theorem, in this section we compile some notions and ideas about sparseness of the regions of intense fluid activity whose mathematical setup was developed in \citet{Grujic2001, Grujic2013} and later has been reformulated and applied for various purposes (see, e.g., \citet{Farhat2017} and \citet{Bradshaw2019}). We also provide a technical lemma in preparation for the proof of the main theorem.

Let $S$ be an open subset of $\bR^3$ and $\mu$ be the Lebesgue measure.
\begin{definition}
For any spatial point $x_0$ and $\delta\in (0,1)$, an open set $S$ is 1D $\delta$-sparse around $x_0$ at scale $r$ if there exists a unit vector $\nu$ such that
\begin{align*}
\frac{\mu\left(S\cap (x_0-r\nu, x_0+r\nu)\right)}{2r} \le \delta\ .
\end{align*}
\end{definition}
The volumetric version is the following.
\begin{definition}
For any spatial point $x_0$ and $\delta\in (0,1)$, an open set $S$ is 3D $\delta$-sparse around $x_0$ at scale $r$ if
\begin{align*}
\frac{\mu\left(S\cap B_r(x_0)\right)}{\mu(B_r(x_0))} \le \delta\ .
\end{align*}
In addition, $S$ is said to be $r$-semi-mixed with ratio $\delta$ if the above inequality holds for every $x_0\in\bR^3$. (It is straightforward to show that for any $S$, 3D $\delta$-sparseness at scale $r$ implies 1D $\delta^{1/3}$-sparseness at scale $r$ around any spatial point $x_0$; however the converse is false.)
\end{definition}

Based on the scale of sparseness of the super-level sets of the positive and negative parts of the vectorial components of a function $f$, \citet{Bradshaw2019} introduced the classes $Z_\alpha$ as a new device for scaling comparison of solutions to the 3D NSE.
\begin{definition}
For $\alpha>0$, $\lambda,\delta\in(0,1)$ and $c_0>1$, we denote by $Z_\alpha(\lambda,\delta,c_0)$ the set of all bounded continuous functions $f:\bR^3\to\bR^3$ satisfying that for any $x_0\in\bR^3$, if $f_j^\pm(x_0)=|f(x_0)|$, the set
\begin{align*}
S_j^\pm:=\left\{x\in\bR^3\ |\ f_j^\pm(x)>\lambda\|f\|_\infty\right\}
\end{align*}
is $\delta$-sparse around $x_0$ at scale $\frac{1}{c}\frac{1}{\|f\|_\infty^\alpha}$ for some $c\in(\frac{1}{c_0},c_0)$.
\end{definition}

Since the sparseness is utilized via the harmonic measure maximum principle for subharmonic functions, the spatial analyticity of solutions is a key player. Here, we recall the result on the spatial analyticity of $u$ and $\omega$, inspired by the method of finding a lower bound on the uniform radius of spatial analyticity of solutions in $L^p$ spaces introduced in \citet{Grujic1998}.

\begin{theorem}[\citet{Guberovic2010} and \citet{Bradshaw2019}]
Let the initial datum $u_0\in L^\infty$ (resp. $\omega_0\in L^\infty\cap L^2$). Then, for any $M>1$, there exists a constant $c(M)$ such that there is a unique mild solution $u$ (resp. $\omega$) in $C_w([0,T], L^\infty)$ where $T\ge\frac{1}{c(M)^2\|u_0\|_\infty^2}$ (resp. $T\ge\frac{1}{c(M)\|\omega_0\|_\infty}$), which has an analytic extension $U(t)$ (resp. $W(t)$) to the region
\begin{align*}
\cD_t:=\left\{x+iy\in\bC^3\ :\ |y|\le \sqrt{t}/c(M)\ \left(\textrm{resp. }|y|\le \sqrt{t}/\sqrt{c(M)}\right)\right\}
\end{align*}
for all $t\in[0,T]$, and
\begin{align*}
\sup_{t\le T}\|U(t)\|_{L^\infty(\cD_t)}\le M\|u_0\|_\infty\ \left(\textrm{resp. }\sup_{t\le T}\|W(t)\|_{L^\infty(\cD_t)}\le M\|\omega_0\|_\infty\right).
\end{align*}
\end{theorem}

As has been demonstrated in \citet{Farhat2017} and \citet{Bradshaw2019}, the concept of `escape time' provides a more optimal 
setting for applying the harmonic measure principle-based sparseness argument.
\begin{definition}
Let $u$ (resp. $\omega$) be in $C((0,T^*), L^\infty)$ where $T^*$ is the first possible blow-up time. A time $t\in(0,T^*)$ is an escape time if $\|u(s)\|_\infty>\|u(t)\|_\infty$ (resp. $\|\omega(s)\|_\infty>\|\omega(t)\|_\infty$) for any $s\in(t,T^*)$. (Local-in-time continuity of $L^\infty$-norm implies there are continuum-many escape times.)
\end{definition}

Next we recall a theorem in \citet{Farhat2017} and an analogous result (a variation of it) in \citet{Bradshaw2019} to be utilized in the proof of our main result,
\begin{theorem}[\citet{Farhat2017} and \citet{Bradshaw2019}]\label{th:SparsityRegVel}
Let $u$ (resp. $\omega$) be in $C((0,T^*), L^\infty)$ where $T^*$ is the first possible blow-up time, and assume, in addition, that $u_0\in L^\infty$ (resp. $\omega_0\in L^\infty\cap L^2$). Let $t$ be an escape time of $u(t)$ (resp. $\omega(t)$), and suppose that there exists a temporal point
\begin{align*}
&\qquad\quad s=s(t)\in \left[t+\frac{1}{4c(M)^2\|u(t)\|_\infty^2},\ t+\frac{1}{c(M)^2\|u(t)\|_\infty^2}\right]
\\
&\left(\textrm{resp. } s=s(t)\in \left[t+\frac{1}{4c(M)\|\omega(t)\|_\infty},\ t+\frac{1}{c(M)\|\omega(t)\|_\infty}\right]\ \right)
\end{align*}
such that for any spatial point $x_0$, there exists a scale $\rho\le \frac{1}{2c(M)^2\|u(s)\|_\infty}$ $\left(\textrm{resp. }\rho\le \frac{1}{2c(M)\|\omega(s)\|_\infty^{\frac{1}{2}}}\right)$ with the property that the super-level set
\begin{align*}
&\qquad\quad V_\lambda^{j,\pm}=\left\{x\in\bR^3\ |\ u_j^\pm(x,s)>\lambda \|u(s)\|_\infty\right\}
\\
&\left(\textrm{resp. }\Omega_\lambda^{j,\pm}=\left\{x\in\bR^3\ |\ \omega_j^\pm(x,s)>\lambda \|\omega(s)\|_\infty\right\}\ \right)
\end{align*}
is 1D $\delta$-sparse around $x_0$ at scale $\rho$; here the index $(j,\pm)$ is chosen such that $|u(x_0,s)|=u_j^\pm(x_0,s)$ (resp. $|\omega(x_0,s)|=\omega_j^\pm(x_0,s)$), and the pair $(\lambda,\delta)$ is chosen such that the followings hold:
\begin{align*}
\lambda h+(1-h)=2\lambda\ ,\qquad h=\frac{2}{\pi}\arcsin\frac{1-\delta^2}{1+\delta^2}\ , \qquad \frac{1}{1+\lambda}<\delta<1\ .
\end{align*}
(Note that such pair exists and a particular example is that when $\delta=\frac{3}{4}$, $\lambda>\frac{1}{3}$.) Then, there exists $\gamma>0$ such that $u\in L^\infty((0,T^*+\gamma); L^\infty)$, i.e. $T^*$ is not a blow-up time.
\end{theorem}

The following lemma is a vector-valued, uniformly-local $L^2$ version of a scalar-valued, Sobolev space lemma in \citet{Iyer2014}; a vectorial $B^{-1}_{\infty,\infty}$ version and $H^{-1}$ version were given respectively in \citet{Farhat2017} and \citet{Bradshaw2019}. 

\begin{lemma}\label{le:3DSparse}
Let $r\in(0,1]$ and $f$ a bounded and continuously differentiable vector-valued function in $\bR^3$. Then, for any pair $(\lambda, \delta)$, $\lambda\in (0,1)$ and $\delta\in(\frac{1}{1+\lambda},1)$, there exists $c^*(\lambda,\delta)>0$ such that if
\begin{align}\label{eq:ZalphaCond}
\sup_{x\in\bR^3} \|f\|_{L^2(B_r(x))} \le c^*(\lambda,\delta) r^{\frac{5}{2}} \|\nabla\times f\|_\infty
\end{align}
then each of the six super-level sets
\begin{align*}
S_\lambda^{i,\pm}=\left\{x\in\bR^3\ |\ (\nabla\times f)_i^\pm(x)>\lambda \|\nabla\times f\|_\infty\right\}\ , \qquad i=1,2,3
\end{align*}
is $(\kappa r)$-semi-mixed with ratio $\delta$, where $\kappa=\sqrt[3]{\frac{\delta(\lambda+1)+1}{2\delta(\lambda+1)}}$.
\end{lemma}

\begin{proof}
Assume the opposite, i.e. there is an index $i$ such that either $S_\lambda^{i,+}$ or $S_\lambda^{i,-}$ 
is not $(\kappa r)$-semi-mixed with the ratio $\delta$. Without loss of generality, suppose it is $S_\lambda^{1,+}$. Then there exists a spatial point $x_0$ such that
\begin{align}\label{eq:OppSparse}
\mu(S_\lambda^{1,+}\cap B_{\kappa r}(x_0)) >  \varpi \delta \kappa^3 r^3
\end{align}
where $\varpi$ denotes the volume of the 3D unit ball. Let $\phi$ be a smooth, radially symmetric and radially decreasing function such that
\begin{align*}
\phi =\left\{\begin{array}{ccc} 1  &\textrm{ on } B_{\kappa r}(x_0) \\ 0 &\textrm{ on } \left(B_r(x_0)\right)^c\end{array}\right. \qquad\textrm{and}\qquad |\nabla \phi|\lesssim (1-\kappa)^{-1}r^{-1}
\end{align*}
First observe that
\begin{align}
\left|\int_{\bR^3} (\nabla\times f)_1(y)\phi(y)dy\right| &=\left|\int_{B_r(x_0)} \left(f_3(y) (\nabla\phi)_2(y)- f_2(y)(\nabla\phi)_3(y)\right) dy\right| \notag
\\
&\le \|f\|_{L^2(B_r(x_0))}\ \|\nabla\phi\|_{L^2(B_r(x_0))} \notag
\\
&\lesssim (1-\kappa)^{\frac{1}{2}}r^{\frac{1}{2}}\|f\|_{L^2(B_r(x_0))} \label{eq:L2locHolder}
\end{align}
To develop a contradictive result to \eqref{eq:ZalphaCond}, we write
\begin{align}\label{eq:DecompSparse}
\left|\int_{\bR^3} (\nabla\times f)_1(y)\phi(y)dy\right| &\ge \int_{\bR^3} (\nabla\times f)_1(y)\phi(y)dy \ge I-J-K
\end{align}
where
\begin{align*}
I &= \int_{S_\lambda^{1,+}\cap B_{\kappa r}(x_0)} (\nabla\times f)_1(y)\phi(y)dy
\\
J &= \left|\int_{B_{\kappa r}(x_0)\setminus S_\lambda^{1,+}} (\nabla\times f)_1(y)\phi(y)dy\right|
\\
K &= \left|\int_{B_r(x_0)\setminus B_{\kappa r}(x_0)} (\nabla\times f)_1(y)\phi(y)dy\right|
\end{align*}
With \eqref{eq:OppSparse} we have estimates
\begin{align*}
I &= \int_{S_\lambda^{1,+}\cap B_{\kappa r}(x_0)} (\nabla\times f)_1^+(y) dy
\\
&> \lambda\|\nabla\times f\|_\infty\  \mu\left(S_\lambda^{1,+}\cap B_{\kappa r}(x_0)\right)
\\
&\ge \varpi\lambda\delta\kappa^3r^3 \|\nabla\times f\|_\infty
\end{align*}
and
\begin{align*}
J &\le \|\nabla\times f\|_\infty\ \mu\left(B_{\kappa r}(x_0)\setminus S_\lambda^{1,+}\right)
\\
&\le \|\nabla\times f\|_\infty \left(\mu\left(B_{\kappa r}(x_0)\right) - \mu\left(B_{\kappa r}(x_0)\cap S_\lambda^{1,+}\right) \right)
\\
&\le \varpi(1-\delta)\kappa^3r^3\|\nabla\times f\|_\infty
\end{align*}
and
\begin{align*}
K &\le \|\nabla\times f\|_\infty \left(\mu\left(B_r(x_0)\right) - \mu\left(B_{\kappa r}(x_0)\right) \right) \le \varpi(1-\kappa^3)r^3\|\nabla\times f\|_\infty
\end{align*}
Combining the above estimates for $I,J,K$, \eqref{eq:L2locHolder} and \eqref{eq:DecompSparse}, we deduce
\begin{align}
(1-\kappa)^{\frac{1}{2}}r^{\frac{1}{2}}\|f\|_{L^2(B_r(x_0))} \gtrsim \varpi r^3 \|\nabla\times f\|_\infty\left((\lambda+1)\delta \kappa^3 -1\right)
\end{align}
in other words, for some constant $c$,
\begin{align}
\|f\|_{L^2(B_r(x_0))} > c \varpi(1-\kappa)^{-\frac{1}{2}} \left((\lambda+1)\delta \kappa^3 -1\right) r^{\frac{5}{2}} \|\nabla\times f\|_\infty
\end{align}
Since $\delta>\frac{1}{1+\lambda}$, if we set $(\lambda+1)\delta \kappa^3 = \frac{\delta(1+\lambda)+1}{2}$, then
\begin{align*}
\|f\|_{L^2(B_r(x_0))} > c^*(\lambda, \delta)\ r^{\frac{5}{2}} \|\nabla\times f\|_\infty
\end{align*}
where $c^*(\lambda, \delta)=c \varpi(1-\kappa)^{-\frac{1}{2}} (\delta(1+\lambda)-1)/2$ with $\kappa=\sqrt[3]{\frac{\delta(\lambda+1)+1}{2\delta(\lambda+1)}}$, which produces a contradiction.
\end{proof}

\section{Pre-duality in Morrey-type Spaces}\label{sec:MorreyPreDual}

In this section, we review some known results about duality and pre-duality of Morrey spaces in a more general setting as well as prove a key lemma in order to fit our arguments into the setting of general Morrey-type spaces $GM_{p,\theta,w}$, with an eye on relaxing the requirements on the lower restriction of the scale $r$ by modifying the parameters $p$ and $\theta$.
\begin{definition}[\citet{Burenkov2004}]
Let $0<p,\ \theta\le \infty$ and let $w$ be a non-negative measurable function on $(0,\infty)$. Denote by $LM_{p,\theta,w}(\bR^d)$ the local Morrey-type space, the space of all functions $f\in L_{loc}^p(\bR^d)$ with finite quasi-norms
\begin{align*}
\|f\|_{LM_{p,\theta,w}(\bR^d)}:= \|w(r)\|f\|_{L^p(B_r(0))}\|_{L^\theta(0,\infty)}\ .
\end{align*}
Denote by $^{\cC}\!LM_{p,\theta,w}(\bR^d)$ the complementary local Morrey-type space, the space of all functions $f\in L^p(\bR^d\setminus B_r(0))$ for all $r>0$ with finite quasi-norms
\begin{align*}
\|f\|_{^{\cC}\!LM_{p,\theta,w}(\bR^d)}:=\|w(r)\|f\|_{L^p(\bR^d\setminus B_r(0))}\|_{L^\theta(0,\infty)}\ .
\end{align*}
A global version of such Morrey-type spaces, denoted by $GM_{p,\theta,w}$, is defined as the space of all functions $f\in L_{loc}^p(\bR^d)$ with finite quasi-norms
\begin{align*}
\|f\|_{GM_{p,\theta,w}}:= \sup_{x\in\bR^d} \|f(x+\cdot)\|_{LM_{p,\theta,w}(\bR^d)}\ .
\end{align*}
\end{definition}

Analogously, global versions of the complementary local Morrey-type spaces are defined as follows.
\begin{definition}[\citet{Burenkov2007} and \citet{Gogatishvili2013}]
Let $0<p,\ \theta\le \infty$ and let $w$ be a non-negative measurable function on $(0,\infty)$. Denote by $\Delta(^{\cC}\!LM_{p,\theta,w})$ the space of all functions $f\in L_{loc}^p(\bR^d)$ with finite quasi-norms
\begin{align*}
\|f\|_{\Delta(^{\cC}\!LM_{p,\theta,w})}:= \sup_{x\in\bR^d} \|f(x+\cdot)\|_{^{\cC}\!LM_{p,\theta,w}(\bR^d)}\ .
\end{align*}
Denote by $^{\cC}\!GM_{p,\theta,w}$ the space of all distributions $f=\sum_k f_k$ where $f_k(x_k+\cdot)\in {}^{\cC}\!LM_{p,\theta,w}(\bR^d)$ and $x_k\in\bR^d$ with finite quasi-norms
\begin{align*}
\|f\|_{^{\cC}\!GM_{p,\theta,w}}:=\inf_{f=\sum_k f_k}\sum_k \|f_k(x_k+\cdot)\|_{^{\cC}\!LM_{p,\theta,w}(\bR^d)}\ .
\end{align*}
\end{definition}

For the dual spaces of these local Morrey-type spaces, \citet{Gogatishvili2011} and \citet{Gogatishvili2013} give partial answers as follow.
\begin{theorem}[\citet{Gogatishvili2011}, Theorem 6.1 and Theorem 6.2]\label{th:DualMorLoc}
Denote by $\Omega_\theta$ the set of all non-negative measurable functions $w$ on $(0,\infty)$ such that
\begin{align*}
0<\|w\|_{L^\theta(r,\infty)}<\infty\ , \qquad r>0
\end{align*}
and by $^{\cC}\!\Omega_\theta$ the set of all non-negative measurable functions $w$ on $(0,\infty)$ such that
\begin{align*}
0<\|w\|_{L^\theta(0,r)}<\infty\ , \qquad r>0\ .
\end{align*}
Assume $1\le p,\ \theta<\infty$. If $w\in \Omega_\theta$ and $\|w\|_{L^\theta(0,\infty)}=\infty$, then
\begin{align*}
(LM_{p,\theta,w})^*&={}^{\cC}\!LM_{p',\theta',\tilde{w}} \
\end{align*}
where $\tilde{w}(t)=w^{\theta-1}(t)\left(\displaystyle \int_t^\infty w^\theta(s)ds\right)^{-1}$. If $w\in {}^{\cC}\!\Omega_\theta$ and $\|w\|_{L^\theta(0,\infty)}=\infty$, then
\begin{align*}
\left({}^{\cC}\!LM_{p,\theta,w}\right)^*&=LM_{p',\theta',\bar{w}}\
\end{align*}
where $\bar{w}=w^{\theta-1}(t)\left(\displaystyle \int_0^t w^\theta(s)ds\right)^{-1}$. The duality is defined under the pairing $\langle f,g\rangle =\displaystyle \int_{\bR^d} fg$.
\end{theorem}

\begin{theorem}[\citet{Gogatishvili2013}, Theorem 6.1]
Let $1\le p<\infty$ and $0<\lambda<d$. Then
\begin{align*}
\left({}^{\cC}\!GM_{p',1,\frac{d-\lambda}{p}-1}(\bR^d)\right)^*=GM_{p,\infty,\frac{\lambda-d}{p}}(\bR^d)
\end{align*}
where $GM_{p,\infty,\nu}:=GM_{p,\infty,w}$ $\left(\textrm{resp. }{}^{\cC}\!GM_{p,\infty,\nu}:={}^{\cC}\!GM_{p,\infty,w}\right)$ with $w(r)=r^\nu$.
\end{theorem}

The results above shall lead us to some proximity of Lemma~\ref{le:3DSparse} since the key step in contradiction argument was the bound on integration against test functions. However, they are not directly applicable because they require either $\|w\|_{L^\theta(0,\infty)}=\infty$ or $w\equiv r^\nu$ while our goal is to restrict the scale $r$ to an interval of the form $[\eta,1]$. 
In order to obtain the duality (or pre-duality) results in $GM_{p,\infty,w}$ with general $w$, we refer to the following.

\begin{theorem}[\citet{Gogatishvili2012}, Theorem~5.4]\label{th:ReverseHardy}
Assume that $0<p\le 1$, $p<q\le \infty$ and $\ell$ is determined by $\frac{1}{\ell}=\frac{1}{p}-\frac{1}{q}$. Let $f$ and $w$ be weight functions on $\bR^d$ and $(0,\infty)$ respectively. Assume $\|w\|_{L^q(t,\infty)} <\infty$ for all $t>0$ and $w\ne 0$ a.e. on $(0,\infty)$. Then, the inequality
\begin{align*}
\|fg\|_{L^p(\bR^d)} \le c\left\|w(t)\int_{B_t(0)}|g(y)|dy\right\|_{L^q(0,\infty)}
\end{align*}
holds for all measurable function $g$ if and only if
\begin{align*}
C_0:= \left(\int_0^\infty \|f\|^{\ell}_{L^{p'}(\bR^d\setminus B_t(0))}\ \frac{d}{dt} \left(\|w\|^{-\ell}_{L^q(t,\infty)}\right)\right)^{1/\ell}\!+ \frac{\|f\|_{L^{p'}(\bR^d)}}{\|w\|_{L^q(0,\infty)}} <\infty
\end{align*}
The sharpest constant $c\approx C_0$.
\end{theorem}

\begin{remark}\label{re:ConvLebStInt}
Following Convention~5.3 in \citet{Gogatishvili2012}, if $\|w\|_{L^q(t,\infty)}=0$ at $t=b$ we set $\displaystyle\frac{d}{dt} \left(\|w\|^{-\ell}_{L^q(t,\infty)}\right)=0$ for all $t>b$.
\end{remark}

\bigskip

Next, following the general idea for proving Corollary~3.2 in \citet{Gogatishvili2011}, we derive the following pre-duality result.

\begin{lemma}\label{le:GblMorreyPreDual}
Let $1\le p<\infty$, $1<\theta\le \infty$. Assume $\|w\|_{L^\theta(r,\infty)} <\infty$ for all $r>0$ and $w\ne 0$ a.e. on $(0,\infty)$. Then
\begin{align}
\sup_{g\in GM_{p,\theta,w}}\! \frac{\displaystyle \int_{\bR^d}|f(x)g(x)|dx}{\|g\|_{GM_{p,\theta,w}}}\lesssim \left(\inf_{x\in\bR^d} \int_0^\infty \|f\|^{\theta'}_{L^{p'}(\bR^d\setminus B_r(x))}\ \frac{d}{dr} \left(\|w\|^{-\theta'}_{L^\theta(r,\infty)}\right)\right)^{1/\theta'}\! + \frac{\|f\|_{L^{p'}(\bR^d)}}{\|w\|_{L^\theta(0,\infty)}}
\end{align}
\end{lemma}

\begin{proof}
By definition,
\begin{align*}
\frac{\displaystyle \int_{\bR^d}|f(x)g(x)|dx}{\|g\|_{GM_{p,\theta,w}}}= \left(\frac{\||f|^p |g|^p\|_{L^{1/p}(\bR^d)}}{\displaystyle\sup_{x\in\bR^d}\left\|w^p(r)\int_{B_r(x)} |g(y)|^pdy\right\|_{L^{\theta/p}(0,\infty)}}\right)^{1/p}\ .
\end{align*}
For any $1\le p<\infty$ and $1<\theta\le \infty$, by Theorem~\ref{th:ReverseHardy}
\begin{align}\label{eq:ThmConst}
\||f(x+\cdot)|^p |g(x+\cdot)|^p\|_{L^{1/p}(\bR^d)}\le C \left\|w^p(r)\int_{B_r(0)} |g(x+y)|^pdy\right\|_{L^{\theta/p}(0,\infty)}
\end{align}
holds for all measurable function $g$ on $\bR^d$ if and only if
\begin{align*}
C^*(x):=&\ \left(\int_0^\infty \||f(x+\cdot)|^p\|^{\theta'/p}_{L^{(1/p)'}(\bR^d\setminus B_r(0))}\ \frac{d}{dr} \left(\|w^p\|^{-\theta'/p}_{L^{\theta/p}(r,\infty)}\right)\right)^{p/\theta'} + \frac{\||f(x+\cdot)|^p\|_{L^{\theta'/p}(\bR^d)}}{\|w^p\|_{L^{\theta/p}(0,\infty)}}
\\
=&\ \left(\int_0^\infty \|f\|^{\theta'}_{L^{p'}(\bR^d\setminus B_r(x))}\ \frac{d}{dr} \left(\|w\|^{-\theta'}_{L^{\theta}(r,\infty)}\right)\right)^{p/\theta'} + \left(\frac{\|f\|_{L^{\theta'}(\bR^d)}}{\|w\|_{L^{\theta}(0,\infty)}}\right)^p <\infty\ .
\end{align*}
The sharpest constant in \eqref{eq:ThmConst} satisfies $C\approx C^*(x)$. Therefore, for all $x\in\bR^d$,
\begin{align*}
\sup_{g\in GM_{p,\theta,w}}\!\frac{\displaystyle \int_{\bR^d}|f(y)g(y)|dy}{\|g\|_{GM_{p,\theta,w}}}\le c_*  \left(\left(\int_0^\infty \|f\|^{\theta'}_{L^{p'}(\bR^d\setminus B_r(x))}\ \frac{d}{dr} \left(\|w\|^{-\theta'}_{L^{\theta}(r,\infty)}\right)\right)^{1/\theta'}\!+ \frac{\|f\|_{L^{\theta'}(\bR^d)}}{\|w\|_{L^{\theta}(0,\infty)}}\right)
\end{align*}
where $c_*$ is a constant which only depends on $p$ and $\theta$.
\end{proof}

The following lemma will serve as the lower restricted local Morrey version of Lemma~\ref{le:3DSparse}.
\begin{lemma}\label{le:3DSparseGblMor}
Let $r\in(0,1]$ and $f$ be a bounded and continuously differentiable vector-valued function in $\bR^3$. Suppose $1\le p<\infty$, $1<\theta\le \infty$ and $w(s)=s^{-\alpha}\1_{[\rho,1]}$ with $\alpha\theta>1$ and $0<\rho\ll 1$. Then, for any pair $(\lambda, \delta)$, $\lambda\in (0,1)$ and $\delta\in(\frac{1}{1+\lambda},1)$, there exists $\varepsilon(\lambda,\delta)>0$ such that if
\begin{align}\label{eq:EpsMorreyCond}
\|f\|_{GM_{p,\theta,w}} \le \left\{\begin{array}{ccc} \varepsilon(\lambda,\delta)\ (r\vee\rho)^{\frac{1-\alpha\theta}{\theta}}\ r^{3+\frac{p'-3}{p'}} \|\nabla\times f\|_\infty\ , & \textrm{if }\ \theta<\infty \\ \varepsilon(\lambda,\delta)\ (r\vee\rho)^{-\alpha}\ r^{3+\frac{p'-3}{p'}} \|\nabla\times f\|_\infty\ , & \textrm{if }\ \theta=\infty \end{array}\right.
\\
\left(\textrm{resp. }\|f\|_{GM_{p,\theta,w}} \le \left\{\begin{array}{ccc} \varepsilon(\lambda,\delta)\ (r\vee\rho)^{\frac{1-\alpha\theta}{\theta}}\ r^{3-\frac{3}{p'}} \|f\|_\infty\ , & \textrm{if }\ \theta<\infty \\ \varepsilon(\lambda,\delta)\ (r\vee\rho)^{-\alpha}\ r^{3-\frac{3}{p'}} \|f\|_\infty\ , & \textrm{if }\ \theta=\infty \end{array}\right.\ \right) \notag
\end{align}
where $r\vee\rho := \max\{r, \rho\}$, then each of the six super-level sets
\begin{align*}
S_\lambda^{i,\pm}=\left\{x\in\bR^3\ |\ (\nabla\times f)_i^\pm(x)>\lambda \|\nabla\times f\|_\infty\right\}\ , \qquad i=1,2,3
\\
\left(\textrm{resp. }\ S_\lambda^{i,\pm}=\left\{x\in\bR^3\ |\ f_i^\pm(x)>\lambda \|f\|_\infty\right\}\ , \qquad i=1,2,3\ \ \right)
\end{align*}
is $r$-semi-mixed with ratio $\delta$.
\end{lemma}

\begin{proof}
Assume the opposite, i.e. there is either $S_\lambda^{i,+}$ or $S_\lambda^{i,-}$ which is not $r$-semi-mixed with the ratio $\delta$. Suppose it is $S_\lambda^{1,+}$. Then there exists a spatial point $x_0$ such that
\begin{align}\label{eq:OppSparseR}
\mu\left(S_\lambda^{1,+}\cap B_{r}(x_0)\right) >  \varpi \delta r^3
\end{align}
where $\varpi$ denotes the volume of the 3D unit ball.

Let $\phi$ be a smooth, radially symmetric and radially decreasing function such that
\begin{align*}
\phi =\left\{\begin{array}{ccc} 1  &\textrm{ on } B_r(x_0) \\ 0 &\textrm{ on } \left(B_{(1+\eta)r}(x_0)\right)^c\end{array}\right. \qquad\textrm{and}\qquad |\nabla \phi|\lesssim \eta^{-1}r^{-1}
\end{align*}
Using an analogous argument to the one in the proof of Lemma~\ref{le:3DSparse}, together with the pre-duality of $GM_{p,\theta,w}$ in Lemma~\ref{le:GblMorreyPreDual}, we obtain
\begin{align}
& \left|\int_{\bR^3} (\nabla\times f)_1(y)\phi(y)dy\right| =\left|\int_{\bR^3} \left(f_3(y) (\nabla\phi)_2(y)- f_2(y)(\nabla\phi)_3(y)\right) dy\right| \notag
\\
&\qquad \lesssim \|f\|_{GM_{p,\theta,w}} \left( \left(\inf_{x\in\bR^3} \int_0^\infty \|\nabla\phi\|^{\theta'}_{L^{p'}(\bR^3\setminus B_t(x))}\ \frac{d}{dt} \left(\|w\|^{-\theta'}_{L^\theta(t,\infty)}\right)\right)^{1/\theta'}\! + \frac{\|\nabla\phi\|_{L^{p'}(\bR^3)}}{\|w\|_{L^\theta(0,\infty)}} \right)\ . \label{eq:MorDualIneq}
\end{align}
Before arguing for a contradiction, we need to compute the quantities $\|\nabla\phi\|_{L^{p'}(\bR^3)}$, $\|w\|_{L^\theta(0,\infty)}$ and the Lebesgue-Stieltjes integral $\displaystyle\int_0^\infty \|\nabla\phi\|^{\theta'}_{L^{p'}(\bR^3\setminus B_t(x))}\ \frac{d}{dt} \left(\|w\|^{-\theta'}_{L^\theta(t,\infty)}\right)$. From the assumption on $\phi$ and $w$, 
\begin{align*}
\|\nabla\phi\|_{L^{p'}(\bR^3\setminus B_t(x_0))} &= \left(\int_{B_{(1+\eta)r}(x_0)\setminus B_t(x_0)} |\nabla \phi|^{p'} dy\right)^{1/p'}
\\
&\lesssim \left\{\begin{array}{ccc} \left(\left(\eta^{-1} r^{-1}\right)^{p'}\cdot \left((1+\eta)^3 - 1\right) r^3\right)^{1/p'} & \textrm{if } t\le r \\ \left(\left(\eta^{-1} r^{-1}\right)^{p'}\cdot \left((1+\eta)^3 - (t/r)^3\right) r^3\right)^{1/p'} & \textrm{if } r<t< (1+\eta)r \\ 0 & \textrm{if } t\ge (1+\eta)r \end{array}\right.\ .
\end{align*}
If $\theta<\infty$, then
\begin{align*}
\|w\|_{L^\theta(t,\infty)} &= \left\{\begin{array}{ccc} \left(\displaystyle \int_\beta^1 s^{-\alpha\theta}ds\right)^{1/\theta}\lesssim \left(\beta^{1-\alpha\theta} -1\right)^{1/\theta} & \textrm{if } 0<t\le\beta \\ \left(\displaystyle \int_t^1 s^{-\alpha\theta}ds\right)^{1/\theta}\lesssim \left(t^{1-\alpha\theta} -1\right)^{1/\theta} & \textrm{if } \beta<t<1 \\ 0 & \textrm{if } t\ge1 \end{array}\right.\ .
\end{align*}
If $\theta=\infty$, then
\begin{align*}
\|w\|_{L^\theta(t,\infty)} &\lesssim \left\{\begin{array}{ccc} \beta^{-\alpha} & \textrm{if } 0<t\le\beta \\ t^{-\alpha} & \textrm{if } \beta<t<1 \\ 0 & \textrm{if } t\ge1 \end{array}\right.\ .
\end{align*}
In the rest of the proof, we only present the case $\theta<\infty$ since the deduction for $\theta=\infty$ is similar and easier. Now, for $\theta<\infty$,
\begin{align}\label{eq:PreDualEst2}
\frac{\|\nabla\phi\|_{L^{p'}(\bR^3)}}{\|w\|_{L^\theta(0,\infty)}} \lesssim \frac{\left((1+\eta)^3-\eta^3\right)^{1/p'} r^{\frac{3}{p'}-1}}{\left(\beta^{1-\alpha\theta} -1\right)^{1/\theta} \eta} \lesssim \eta^{-1} \left((1+\eta)^3-1\right)^{1/p'} \beta^{(\alpha\theta-1)/\theta} r^{\frac{3}{p'}-1}\ .
\end{align}
Note that the Lebesgue-Stieltjes integral in \eqref{eq:MorDualIneq} attains its minimum when $x=x_0$ (in which case $\bR^3\setminus B_t(x)$ covers most of the support of $\nabla \phi$), and referring to Remark~\ref{re:ConvLebStInt} we write
\begin{align*}
\frac{d}{dt} \left(\|w\|^{-\theta'}_{L^\theta(t,\infty)}\right) \lesssim \left\{\begin{array}{ccc} 0 & \textrm{if } 0<t\le\beta \\ \left(t^{1-\alpha\theta} -1\right)^{-\theta'/\theta-1}t^{-\alpha\theta} & \textrm{if } \beta<t<1 \\ 0 & \textrm{if } t\ge1 \end{array}\right.\ .
\end{align*}
Therefore, combining the above calculations gives the estimate
\begin{align*}
&\inf_{x\in\bR^3} \int_0^\infty \|\nabla\phi\|^{\theta'}_{L^{p'}(\bR^3\setminus B_t(x))}\ \frac{d}{dt} \left(\|w\|^{-\theta'}_{L^\theta(t,\infty)}\right)
\\
&\lesssim \left\{\begin{array}{ccc} \begin{array}{cc} \displaystyle \int_r^{(1+\eta)r}\left(\left((1+\eta)^3-\left(t/r\right)^3\right)^{\frac{1}{p'}} \eta^{-1}\right)^{\theta'} r^{(\frac{3}{p'}-1)\theta'}\left(t^{1-\alpha\theta} -1\right)^{-\frac{\theta'}{\theta}-1}t^{-\alpha\theta}dt \\ \displaystyle +\int_\beta^r\left(\left((1+\eta)^3-1\right)^{\frac{1}{p'}} \eta^{-1}\right)^{\theta'} r^{(\frac{3}{p'}-1)\theta'}\left(t^{1-\alpha\theta} -1\right)^{-\frac{\theta'}{\theta}-1}t^{-\alpha\theta}dt \end{array} & \textrm{if } \beta\le r \\ \displaystyle \int_\beta^{(1+\eta)r}\left(\left((1+\eta)^3-\left(t/r\right)^3\right)^{\frac{1}{p'}} \eta^{-1}\right)^{\theta'} r^{(\frac{3}{p'}-1)\theta'}\left(t^{1-\alpha\theta} -1\right)^{-\frac{\theta'}{\theta}-1}t^{-\alpha\theta}dt & \textrm{if } r<\beta<(1+\eta)r \\ 0 & \textrm{if } \beta\ge (1+\eta)r \end{array}\right.\ .
\end{align*}
Simplification via some asymptotic properties reduces the estimate to
\begin{align}
&\left(\inf_{x\in\bR^3} \int_0^\infty \|\nabla\phi\|^{\theta'}_{L^{p'}(\bR^3\setminus B_t(x))}\ \frac{d}{dt} \left(\|w\|^{-\theta'}_{L^\theta(t,\infty)}\right)\right)^{1/\theta'} \notag
\\
&\lesssim \left\{\begin{array}{ccc} \eta^{-1}\left((1+\eta)^3 - 1\right)^{\frac{1}{p'}} \left((1+\eta)^{\frac{(\alpha\theta -1)\theta'}{\theta}} - 1\right)^{\frac{1}{\theta'}} r^{\frac{3}{p'}-1} \left(r\vee \beta\right)^{\frac{\alpha\theta -1}{\theta}} & \textrm{if } \beta< (1+\eta)r  \\ 0 & \textrm{if } \beta\ge (1+\eta)r \end{array}\right.\ . \label{eq:PreDualEst1}
\end{align}
Now, following the idea in the proof of Lemma~\ref{le:3DSparse}, we write
\begin{align}\label{eq:DecompSparse}
\left|\int_{\bR^3} (\nabla\times f)_1(y)\phi(y)dy\right| &\ge \int_{\bR^3} (\nabla\times f)_1(y)\phi(y)dy \ge I-J-K
\end{align}
where
\begin{align*}
I &= \int_{S_\lambda^{1,+}\cap B_r(x_0)} (\nabla\times f)_1(y)\phi(y)dy
\\
J &= \left|\int_{B_r(x_0)\setminus S_\lambda^{1,+}} (\nabla\times f)_1(y)\phi(y)dy\right|
\\
K &= \left|\int_{B_{(1+\eta)r}(x_0)\setminus B_r(x_0)} (\nabla\times f)_1(y)\phi(y)dy\right|
\end{align*}
Similarly, with \eqref{eq:OppSparseR} we deduce
\begin{align*}
I &> \lambda\|\nabla\times f\|_\infty\  \mu\left(S_\lambda^{1,+}\cap B_{r}(x_0)\right)
\ge \varpi\lambda\delta r^3 \|\nabla\times f\|_\infty
\\
J &\le \|\nabla\times f\|_\infty \left(\mu\left(B_{r}(x_0)\right) - \mu\left(B_{r}(x_0)\cap S_\lambda^{1,+}\right) \right) \le \varpi(1-\delta) r^3\|\nabla\times f\|_\infty
\end{align*}
and
\begin{align*}
K &\le \|\nabla\times f\|_\infty \left(\mu\left(B_{(1+\eta)r}(x_0)\right) - \mu\left(B_{r}(x_0)\right) \right) \le \varpi((1+\eta)^3-1)r^3\|\nabla\times f\|_\infty\ .
\end{align*}
Combining the estimates for $I,J,K$, \eqref{eq:MorDualIneq}-\eqref{eq:DecompSparse}, we deduce
\begin{align*}
\varpi r^3 \left((1+\lambda)\delta-(1+\eta)^3\right) & \|\nabla\times f\|_\infty \lesssim \|f\|_{GM_{p,\theta,w}} \left(\eta^{-1} \left((1+\eta)^3-1\right)^{\frac{1}{p'}} \beta^{\frac{\alpha\theta-1}{\theta}} r^{\frac{3}{p'}-1} \right.
\\
& \left. +\ \eta^{-1} \left((1+\eta)^3 - 1\right)^{\frac{1}{p'}} \left((1+\eta)^{\frac{(\alpha\theta -1)\theta'}{\theta}} - 1\right)^{\frac{1}{\theta'}} r^{\frac{3}{p'}-1} \left(r\vee \beta\right)^{\frac{\alpha\theta -1}{\theta}}\right)
\\
&\qquad \lesssim \|f\|_{GM_{p,\theta,w}}\ \eta^{-1} \left((1+\eta)^3 - 1\right)^{\frac{1}{p'}} (1+\eta)^{\frac{\alpha\theta -1}{\theta}} r^{\frac{3}{p'}-1} \left(r\vee \beta\right)^{\frac{\alpha\theta -1}{\theta}}\ ,
\end{align*}
in other words, for some constant $c$,
\begin{align}
\|f\|_{GM_{p,\theta,w}} > \frac{c\ \varpi\left((1+\lambda)\delta-(1+\eta)^3\right) r^{3+\frac{p'-3}{p'}} (r\vee\beta)^{\frac{1-\alpha\theta}{\theta}} }{\eta^{-1} \left((1+\eta)^3 - 1\right)^{\frac{1}{p'}} (1+\eta)^{\frac{\alpha\theta -1}{\theta}}}\ \|\nabla\times f\|_\infty\ .
\end{align}
Since $\delta>\frac{1}{1+\lambda}$, if we set $(1+\eta)^3 = \frac{\delta(1+\lambda)+1}{2}$, then
\begin{align*}
\|f\|_{GM_{p,\theta,w}} > \varepsilon(\lambda, \delta)\ r^{3+\frac{p'-3}{p'}} (r\vee\beta)^{\frac{1-\alpha\theta}{\theta}} \|\nabla\times f\|_\infty
\end{align*}
where $\varepsilon(\lambda, \delta)=c\ \varpi\left(\frac{\delta(1+\lambda)-1}{2}\right)^{1-\frac{1}{p'}} \left(\frac{\delta(1+\lambda)+1}{2}\right)^{\frac{1-\alpha\theta}{3\theta}} \left(\sqrt[3]{\frac{\delta(1+\lambda)+1}{2}} -1\right)$, which produces a contradiction.

\end{proof}

\section{The Main Results}\label{sec:MainThm}

The main results are the following.
\begin{theorem}\label{th:MainThm}
Let $u$ be a unique regular solution to the 3D NSE evolving from $u_0\in L^\infty$ and $T^*$ be the first possible blow-up time. Assume that the initial value of the vorticity $\omega_0\in L^2\cap L^\infty$ and let $t$ be an escape time. There exists $\epsilon_0$ such that if the solution $u$ and the vorticity $\omega$ satisfy
\begin{align}\label{eq:RestrMorreyCond}
\inf_{s\in\left[t+\frac{1}{4c_0\|\omega(t)\|_\infty},\ t+\frac{1}{c_0\|\omega(t)\|_\infty}\right]}\quad \sup_{x\in\bR^3,\ \|\omega(s)\|_\infty^{-\frac{1}{2}}\lesssim r\le 1} \quad \frac{1}{r} \int_{B_r(x)} |u(s,y)|^2dy\ \le\ \epsilon_0\ ,
\end{align}
then, there exists $\gamma>0$ such that $u\in L^\infty((0,T^*+\gamma); L^\infty)$, i.e. $T^*$ is not a blow-up time.
\end{theorem}

\begin{proof}
By the assumption there exists $s\in\left[t+\frac{1}{4c_0\|\omega(t)\|_\infty},\ t+\frac{1}{c_0\|\omega(t)\|_\infty} \right]$ such that
\begin{align*}
\sup_{x\in\bR^3,\ \|\omega(s)\|_\infty^{-\frac{1}{2}}\lesssim r\le 1} \ r^{-1} \|u(s)\|^2_{L^2(B_r(x))}\ \le\ \epsilon_0\ .
\end{align*}
This can be reformulated as
\begin{align*}
\sup_{x\in\bR^3}r^{-1} \|u(s)\|^2_{L^2(B_r(x))}\ \le\ \epsilon_0 \qquad\textrm{for all }\ \|\omega(s)\|_\infty^{-\frac{1}{2}}\lesssim r\le 1\ ,
\end{align*}
in other words,
\begin{align*}
r^{-\frac{1}{2}}\sup_{x\in\bR^3}\|u(s)\|_{L^2(B_r(x))} \le \epsilon_0^{\frac{1}{2}} \qquad\textrm{for all $r$ such that }\ 1\lesssim \|\omega(s)\|_\infty^{\frac{1}{2}}\ r\le \|\omega(s)\|_\infty^{\frac{1}{2}}\ ,
\end{align*}
which implies that, for all $\|\omega(s)\|_\infty^{-\frac{1}{2}}\lesssim r\le 1$,
\begin{align*}
\sup_{x\in\bR^3}\|u(s)\|_{L^2(B_r(x))} \le \epsilon_0^{\frac{1}{2}}r^{\frac{1}{2}}\lesssim \epsilon_0^{\frac{1}{2}}r^{\frac{1}{2}} \left(\|\omega(s)\|_\infty^{\frac{1}{2}} r\right)^2 \lesssim \epsilon_0^{\frac{1}{2}}r^{\frac{5}{2}}\|\omega(s)\|_\infty\ .
\end{align*}
In particular, for some $r\le\frac{1}{c_0\|\omega(s)\|_\infty^{\frac{1}{2}}}$ we have
\begin{align}
\sup_{x\in\bR^3}\|u(s)\|_{L^2(B_r(x))} \lesssim \epsilon_0^{\frac{1}{2}}r^{\frac{5}{2}}\|\omega(s)\|_\infty\ .
\end{align}
Now applying Lemma~\ref{le:3DSparse} with $f=u(s)$ we can conclude that for any pair $(\lambda, \delta)$ with $\lambda\in(0,1)$ and $\delta\in(\frac{1}{1+\lambda},1)$, there exists sufficiently small $\epsilon_0$ such that all the super-level sets
\begin{align*}
\Omega_\lambda^{i,\pm}=\left\{x\in\bR^3\ |\ \omega_i^\pm(x,s)>\lambda \|\omega(s)\|_\infty\right\}\ , \qquad i=1,2,3
\end{align*}
are $(\kappa r)$-semi-mixed with ratio $\delta$ for some $\kappa<1$. This implies each $\Omega_\lambda^{i,\pm}$ is 1D $(\kappa r)^{\frac{1}{3}}$-sparse around any spatial point $x_0$ at scale $r\le\frac{1}{c_0\|\omega(s)\|_\infty^{\frac{1}{2}}}$, which fits the assumption in Theorem~\ref{th:SparsityRegVel}; therefore $T^*$ is not a blow-up time.
\end{proof}

\bigskip

Using Theorem~\ref{th:SparsityRegVel} together with Lemma~\ref{le:3DSparseGblMor}, we provide a different proof of Theorem~\ref{th:MainThm} which not only extends the result into the general Morrey-type spaces introduced in Section~\ref{sec:MorreyPreDual} but also embeds our regularity result into a more general formalism, the $Z_\alpha$ classes (cf. Section~\ref{sec:SparseRegAssm}).

\begin{theorem}\label{th:MainThmVor}
Let $u$ be a unique regular solution to the 3D NSE evolving from $u_0\in L^\infty$ and $T^*$ be the first possible blow-up time. Assume that the initial value of the vorticity $\omega_0\in L^2\cap L^\infty$ and let $t$ be an escape time. There exists $\epsilon_0$ such that if there is $s\in\left[t+\frac{1}{4c_0\|\omega(t)\|_\infty},\ t+\frac{1}{c_0\|\omega(t)\|_\infty}\right]$ with
\begin{align}\label{eq:RestrMorreyCondVar}
\|u(s)\|_{GM_{p,\theta,w}} \ \le\left\{\begin{array}{ccc} \epsilon_0 \left(\|\omega(s)\|_\infty\right)^{(\alpha\wedge\beta) \frac{\nu\theta-1}{\theta}-\alpha\left(3+\frac{p'-3}{p'}\right)+1} & \textrm{if } 1<\theta<\infty \\ \epsilon_0 \left(\|\omega(s)\|_\infty\right)^{(\alpha\wedge\beta) \nu-\alpha\left(3+\frac{p'-3}{p'}\right)+1} & \textrm{if } \theta=\infty \end{array}\right.
\end{align}
where $w(r)=r^{-\nu}\1_{\left[c\|\omega(s)\|_\infty^{-\beta} ,\ 1\right]}$, then $\omega(s)\in Z_\alpha(\lambda,\delta,\varepsilon^*)$ for any pair $(\lambda, \delta)$ such that $\lambda\in (0,1)$ and $\delta\in(\frac{1}{1+\lambda},1)$ with some $\varepsilon^*$ depending on $(\lambda,\delta)$; in particular, this implies that when $\alpha=\frac{1}{2}$, $T^*$ is not a blow-up time.
If the above parameters satisfy
\begin{center}
$(\alpha\wedge\beta) \frac{\nu\theta-1}{\theta}-\alpha\left(3+\frac{p'-3}{p'}\right)+1=0\ $ if $\theta<\infty$;
\\
$(\alpha\wedge\beta) \nu-\alpha\left(3+\frac{p'-3}{p'}\right)+1=0\ $ if $\theta=\infty$,
\end{center}
the condition \eqref{eq:RestrMorreyCondVar} simply writes $\|u(s)\|_{GM_{p,\theta,w}}\le\epsilon_0$. Note that Theorem~\ref{th:MainThm} is the special case when $\nu=\beta=\frac{1}{2}$, $p=2$ and $\theta=\infty$.
\end{theorem}

\begin{proof}
This is an immediate consequence of the first part of Lemma~\ref{le:3DSparseGblMor}, and when $\alpha=\frac{1}{2}$ the continuation of smooth solution at $T^*$ follows from Theorem~\ref{th:SparsityRegVel}.
\end{proof}

\begin{theorem}\label{th:MainThmVel}
Let $u$ be a unique regular solution to the 3D NSE evolving from $u_0\in L^\infty$ and $T^*$ be the first possible blow-up time. Assume that the initial value of the vorticity $\omega_0\in L^2\cap L^\infty$ and let $t$ be an escape time. There exists $\epsilon_0$ such that if there is $s\in\left[t+\frac{1}{4c_0^2\|u(t)\|_\infty^2},\ t+\frac{1}{c_0^2\|u(t)\|_\infty^2}\right]$ $\left(\textrm{resp. }s\in\left[t+\frac{1}{4c_0^2\|\omega(t)\|_\infty},\ t+\frac{1}{c_0^2\|\omega(t)\|_\infty}\right]\ \right)$ with
\begin{align}\label{eq:RestrMorreyCondVarVel}
&\qquad\quad \|u(s)\|_{GM_{p,\theta,w}} \ \le\left\{\begin{array}{ccc} \epsilon_0 \left(\|u(s)\|_\infty\right)^{(\alpha\wedge\beta) \frac{\nu\theta-1}{\theta}-\alpha\left(3-\frac{3}{p'}\right)+1} & \textrm{if } 1<\theta<\infty \\ \epsilon_0 \left(\|u(s)\|_\infty\right)^{(\alpha\wedge\beta) \nu-\alpha\left(3-\frac{3}{p'}\right)+1} & \textrm{if } \theta=\infty \end{array}\right.
\\
&\left(\textrm{resp. }\|\omega(s)\|_{GM_{p,\theta,w}} \ \le\left\{\begin{array}{ccc} \epsilon_0 \left(\|\omega(s)\|_\infty\right)^{(\alpha\wedge\beta) \frac{\nu\theta-1}{\theta}-\alpha\left(3-\frac{3}{p'}\right)+1} & \textrm{if } 1<\theta<\infty \\ \epsilon_0 \left(\|\omega(s)\|_\infty\right)^{(\alpha\wedge\beta) \nu-\alpha\left(3-\frac{3}{p'}\right)+1} & \textrm{if } \theta=\infty \end{array}\right.\ \right)  \notag
\end{align}
where $w(r)=r^{-\nu}\1_{\left[c\|u(s)\|_\infty^{-\beta} ,\ 1\right]}$ $\left(\textrm{resp. }w(r)=r^{-\nu}\1_{\left[c\|\omega(s)\|_\infty^{-\beta} ,\ 1\right]} \right)$,
then $u(s)\in Z_\alpha(\lambda,\delta,\varepsilon^*)$ (resp. $\omega(s)\in Z_\alpha(\lambda,\delta,\varepsilon^*)$) for any pair $(\lambda, \delta)$ such that $\lambda\in (0,1)$ and $\delta\in(\frac{1}{1+\lambda},1)$ with some $\varepsilon^*$ depending on $(\lambda,\delta)$, which implies that when $\alpha=1$ (resp. $\alpha=\frac{1}{2}$), $T^*$ is not a blow-up time.
In particular, if the above parameters satisfy
\begin{center}
$(\alpha\wedge\beta) \frac{\nu\theta-1}{\theta}-\alpha\left(3-\frac{3}{p'}\right)+1=0\ $ if $\theta<\infty$;
\\
$(\alpha\wedge\beta) \nu-\alpha\left(3-\frac{3}{p'}\right)+1=0\ $ if $\theta=\infty$,
\end{center}
the condition \eqref{eq:RestrMorreyCondVarVel} simply writes $\|u(s)\|_{GM_{p,\theta,w}}\le\epsilon_0$ $\left(\textrm{resp. }\|\omega(s)\|_{GM_{p,\theta,w}}\le\epsilon_0\ \right)$.
\end{theorem}

\begin{proof}
This is an immediate consequence of the second part of Lemma~\ref{le:3DSparseGblMor}, and when $\alpha=1$ (resp. $\alpha=\frac{1}{2}$) the continuation of smooth solution at $T^*$ is a corollary of Theorem~\ref{th:SparsityRegVel}.
\end{proof}

Analogously, one can formulate regularity criteria such as
\begin{align*}
\|u(s)\|_{GM_{p,\theta,w}}\le \epsilon_0 \left(\|u(s)\|_\infty\right)^{\gamma_1}\left(\|\omega(s)\|_\infty\right)^{\gamma_2} \qquad\textrm{with}\quad w(r)=r^{-\nu}\1_{\left[c\|u(s)\|_\infty^{-\beta_1}\|\omega(s)\|_\infty^{-\beta_2} ,\ 1\right]}
\end{align*}
or
\begin{align*}
\|\omega(s)\|_{GM_{p,\theta,w}}\le \epsilon_0 \left(\|u(s)\|_\infty\right)^{\gamma_1}\left(\|\omega(s)\|_\infty\right)^{\gamma_2} \qquad\textrm{with}\quad w(r)=r^{-\nu}\1_{\left[c\|u(s)\|_\infty^{-\beta_1}\|\omega(s)\|_\infty^{-\beta_2} ,\ 1\right]}
\end{align*}
The proofs are replicable.


\def\cprime{$'$}

\end{document}